\documentclass[12pt,a4paper]{amsart}
\usepackage{amsfonts,color}
\usepackage{amsthm}
\usepackage{amsmath}
\usepackage{amscd}
\usepackage[utf8]{inputenc}
\usepackage{t1enc}
\usepackage[mathscr]{eucal}
\usepackage{indentfirst}
\usepackage{graphicx}
\usepackage{graphics}
\usepackage{pict2e}
\usepackage{epic}
\usepackage{url}
\usepackage{epstopdf}
\usepackage{comment}
\usepackage{todonotes}
\usepackage{hyperref}
\usepackage{subcaption}

\numberwithin{equation}{section}
\usepackage[margin=2.6cm]{geometry}

\usepackage{pgfplots}
\usepackage{xcolor}
\usepackage{tikz}
\usetikzlibrary{matrix,arrows,decorations.pathmorphing}
\usetikzlibrary{calc,decorations.pathreplacing}
\usetikzlibrary{quotes,angles}
\usetikzlibrary{shapes}
\usetikzlibrary{patterns}
\usetikzlibrary{arrows.meta}

\tikzstyle{vertex}=[draw=black,circle,fill=black,minimum size=6pt, inner sep=0pt, outer sep=0pt,text=black,line width=0mm]

\tikzstyle{Sqvertex}=[draw=black,shape=rectangle, minimum size=10pt, fill=white]

\tikzstyle{Cvertex}=[draw=black,shape=circle, minimum size=6pt, fill=white]

\tikzstyle{vertex_blue}=[draw=black,circle,fill=blue,minimum size=6pt, inner sep=0pt, outer sep=0pt,text=black,line width=0mm]
\tikzstyle{vertex_red}=[draw=black,circle,fill=red,minimum size=6pt, inner sep=0pt, outer sep=0pt,text=black,line width=0mm]
\tikzstyle{vertex_green}=[draw=black,circle,fill=green,minimum size=6pt, inner sep=0pt, outer sep=0pt,text=black,line width=0mm]
\tikzstyle{my_arrow}=
[shape=arrow,length=10pt]
\tikzstyle{c0}=[shape=circle, minimum size=4pt, fill=white]
\tikzstyle{c1}=[shape=rectangle, minimum size=7pt, fill=red]
\tikzstyle{c2}=[shape=diamond, minimum size=10pt, fill=blue]
\tikzstyle{mybox} = [rectangle, rounded corners, minimum width=3cm, minimum height=1cm,text centered, draw=black]
\tikzset{base/.style = {rectangle, rounded corners, draw=black,
                           minimum width=3cm, minimum height=1cm,
                           text centered}}

\usepgfplotslibrary{fillbetween}
\pgfplotsset{mystyle/.style={%
        xmin=-2,
        xmax=7.9,
        ymin=-1,
        xtick = {1,3},
        xticklabels = {{1},$d-1$},
        ytick = {1}
    }
}

\definecolor{darkerblue}{HTML}{065A82} 
\definecolor{lighterblue}{HTML}{1C7293} 

\theoremstyle{plain}
\newtheorem{Th}{Theorem}[section]

 \theoremstyle{definition}
\newtheorem{Def}[Th]{Definition}

\newtheorem{Rem}[Th]{Remark}
\newtheorem{?}[Th]{Problem}
\newtheorem{Ex}[Th]{Example}

\newcommand{\ee}{\varepsilon}

\newcommand{\x}{\underline{x}}

\begin{document}

\title[Eulerian orientations for Benjamini--Schramm convergent graph sequences]{Number of Eulerian orientations for Benjamini--Schramm convergent graph sequences}

\author[F. Bencs]{Ferenc Bencs}

\address{Centrum Wiskunde \& Informatica, P.O. Box 94079 1090 GB Amsterdam, The Netherlands.}

\email{ferenc.bencs@gmail.com}

\author[M. Borb\'enyi]{M\'arton Borb\'enyi}

\address{ELTE: E\"{o}tv\"{o}s Lor\'{a}nd University  Mathematics Institute, H-1117 Budapest, 
P\'{a}zm\'{a}ny P\'{e}ter s\'{e}t\'{a}ny 1/C \and HUN-REN Alfr\'ed R\'enyi Institute of Mathematics, H-1053 Budapest Re\'altanoda utca 13-15}

\email{marton.borbenyi@tkk.elte.hu}

\author[P. Csikv\'ari]{P\'{e}ter Csikv\'{a}ri}

\address{HUN-REN Alfr\'ed R\'enyi Institute of Mathematics, H-1053 Budapest Re\'altanoda utca 13-15 \and ELTE: E\"{o}tv\"{o}s Lor\'{a}nd University  Mathematics Institute, Department of Computer
Science  H-1117 Budapest, 
P\'{a}zm\'{a}ny P\'{e}ter s\'{e}t\'{a}ny 1/C}

\email{peter.csikvari@gmail.com}

\thanks{The first author was supported by the Netherlands Organisation of Scientific Research (NWO): VI.Veni.222.303. The second author was supported by the \'{U}NKP-23-3 New National Excellence Program of the Ministry for Culture and Innovation from the source of the National Research, Development and Innovation Fund (NKFIH). The research was supported by the MTA-R\'enyi Counting in Sparse Graphs ''Momentum'' Research
Group, and by Dynasnet European Research Council Synergy project -- grant number ERC-2018-SYG 810115.}

\begin{abstract}
For a graph $G$ let $\ee(G)$ denote the number of Eulerian orientations, and $v(G)$ denote the number of vertices of $G$. We show that if $(G_n)_n$ is a sequence of Eulerian graphs that are convergent in Benjamini--Schramm sense, then $\lim_{n\to \infty}\frac{1}{v(G_n)}\ln \ee(G_n)$ is convergent.
\end{abstract}

\maketitle

\section{Eulerian orientations}

A graph $G$ is called Eulerian if every vertex degree is an even number. (In general, connectedness of $G$ is also a requirement, but in this paper we do not need this assumption.) It is a classical fact that the edges of an Eulerian graph can be oriented in such a way that at every vertex the in-degree and the out-degree are equal. Such an orientation is called Eulerian or balanced orientation. The number of Eulerian oriantations is denoted by $\varepsilon(G)$. Counting Eulerian orientations has triggered considerable interest both in combinatorics, computer science and statistical physics. Probably, the best known result is due to Lieb \cite{lieb2004residual} who determined the asymptotic number of Eulerian orientations of large square grid graphs. In physics the limit value is called the residual entropy or the entropy of the ice model. Baxter \cite{baxter1969f} determined the residual entropy for the large triangular lattices. Welsh \cite{welsh1999tutte} observed that for a $4$--regular graph the Tutte-polynomial evaluation $|T_G(0,-2)|$  is exactly the number of Eulerian orientations since nowhere-zero $Z_3$-flows and Eulerian orientations are in one-to-one correspondence for $4$--regular graphs. In this paper we focus on bounded degree graphs, but there has been many advances on asymptotic enumeration of Eulerian orientations in case of non-bounded degree graphs too, see the papers 
\cite{isaev2011asymptotic,isaev2023cumulant,isaev2024correlation,isaev2013asymptotic,mckay1990asymptotic,mckay1998asymptotic}. Mihail and Winkler \cite{mihail1996number} gave an efficient randomized algorithm to sample and approximately count Eulerian orientations. Schrijver \cite{schrijver1983bounds} gave a lower bound for the number of Eulerian oriantations in terms of the degree sequence. He proved that if $G$ is a  graph on $n$ vertices with degree sequence $d_1,d_2,\dots ,d_n$, where $d_k$ are even for all $k$, then
$$\varepsilon(G)\geq \prod_{k=1}^n\frac{\binom{d_k}{d_k/2}}{2^{d_k/2}}.$$
In particular, for a $d$--regular graph $G$ on $n$ vertices, where $d$ is even, we have
$$\varepsilon(G)\geq \left(\frac{\binom{d}{d/2}}{2^{d/2}}\right)^n.$$
The right hand side of this inequality coincides with Pauling's original heuristic argument for the entropy of ice \cite{pauling1935structure}. This heuristic argument is based on the idea that at each vertex the probability that a random orientation is balanced is exactly $\frac{\binom{d}{d/2}}{2^d}$. Assuming an (asymptotic) independence for the vertices we get an estimate for the number of Eulerian orientations. Vergnas \cite{vergnas1990upper} proved an upper bound that has the following corollary. If $(G_n)_n$ is a sequence of $d$-regular graphs such that the length of the shortest cycle, denoted by $g(G_n)$ hereafter, tends to infinity, then
$$\lim_{n\to \infty}\frac{1}{v(G_n)}\ln \ee(G_n)=\ln \left(\frac{\binom{d}{d/2}}{2^{d/2}}\right),$$
where $v(G)$ denotes the number of vertices of a graph $G$. The aforementioned result of Lieb gives that if $(G_n)_n$ is a a sequence of toroidal grids, then
$$\lim_{n\to \infty}\frac{1}{v(G_n)}\ln \ee(G_n)=\frac{3}{2}\ln \left(\frac{4}{3}\right),$$
and the result of Baxter shows that if $(G_n)_n$ is a a sequence of triangular graphs with helical boundary condition, then
$$\lim_{n\to \infty}\frac{1}{v(G_n)}\ln \ee(G_n)=\ln \left(\frac{3\sqrt{3}}{2}\right).$$
To put the results of Lieb, Baxter and Vergnas into a common framework we need the concept of Benjamini--Schramm convergence. This concept grasps graph sequences that are locally look alike.

\begin{Def}[Benjamini--Schramm convergence]
We say that a graph sequence $(G_n)_n$ is \emph{bounded-degree} if there is a $\Delta$ such that the maximum degree of any $G_n$ is at most $\Delta$.

 For a finite graph $G$, a finite connected rooted graph $\alpha$ and a positive integer
$r$, let $\mathbb{P}(G,\alpha,r)$ be the probability that the $r$-ball
centered at a uniform random vertex of $G$ is isomorphic to $\alpha$. 

 Let $L$ be a probability distribution on (finite and infinite) connected rooted graphs; we will call $L$ a \emph{random rooted graph}.
For a finite connected rooted graph $\alpha$ and a positive integer $r$, let $\mathbb{P}(L,\alpha,r)$ be the probability that the $r$-ball
centered at the root vertex  is isomorphic to $\alpha$, where the root is chosen from the distribution $L$.

We say that a bounded-degree graph sequence $(G_n)_n$ is \emph{Benjamini--Schramm
convergent} if for all finite rooted graphs $\alpha$ and $r>0$, the
probabilities $\mathbb{P}(G_n,\alpha,r)$ converge. Furthermore, we say that \emph{$(G_n)$ Benjamini-Schramm converges to $L$},
if for all positive integers $r$ and finite rooted graphs $\alpha$, $\mathbb{P}(G_n,\alpha,r)\rightarrow \mathbb{P}(L,\alpha,r)$.
\end{Def}

The Benjamini--Schramm convergence is also called \emph{local convergence} as it primarily grasps the local structure of the graphs $(G_n)_n$.

If we take larger and larger boxes in the $d$-dimensional grid $\mathbb{Z}^d$, then it will converge to the rooted $\mathbb{Z}^d$, that is,  the corresponding random rooted graph $L$ is the distribution which takes a rooted  $\mathbb{Z}^d$ with probability $1$. (See Figure~\ref{fig:Z2} for two examples of graph sequences that converge to $\mathbb{Z}^2$.)
\begin{figure}[h!]
    \centering
    \begin{subfigure}{.5\textwidth}
    \centering
    \begin{tikzpicture}[scale=0.6]
    
         \foreach \x in {0,1,2,3,4,5,6}
            {
                \draw[thick] (\x,0) -- (\x,6);
                \draw[thick] (0,\x) -- (6,\x);
                \draw[thick, bend right=12,opacity=0.7] (\x,0) to (\x,6);
                \draw[thick, bend left=12,opacity=0.7] (0,\x) to (6,\x);
                
                \foreach \y in {0,1,2,3,4,5,6}{
                    \filldraw[black] (\x,\y) circle (3pt);
                }
            }        
    \end{tikzpicture}
    \end{subfigure}%
    \begin{subfigure}{0.5\textwidth}
    \centering
    \begin{tikzpicture}[scale=0.3]

         \foreach \x in {0,1,2,3}
        {
            \draw[thick] (7-2*\x,2*\x+1) -- (2*\x-7,2*\x+1);
            
            \draw[thick] (2*\x+1,2*\x-7) -- (2*\x+1,7-2*\x);
        
            \draw[thick] (7-2*\x,-2*\x-1) -- (2*\x-7,-2*\x-1);
    
            \draw[thick] (-2*\x-1,2*\x-7) -- (-2*\x-1,7-2*\x);
            
        }
        
        \foreach \x in {0,1,2,3,4}{
            \foreach \y in {0,1,2,3}{
                \filldraw[black] (-1+2*\x-2*\y,-7+2*\x+2*\y) circle (6pt);
                \filldraw[black] (1-2*\x+2*\y,-7+2*\x+2*\y) circle (6pt);
                
            }
        }
    \end{tikzpicture}
    \end{subfigure}
    \caption{In the picture we depict two graph sequences both converging to  $\mathbb{Z}^2$. The latter sequence consists of the so-called Aztec diamonds. Aztec diamonds are Eulerian graphs just as the toroidal grids. As we will see this implies that the normalized number of Eulerian orientations converge to the same number, that is $\frac{3}{2}\ln \left(\frac{4}{3}\right)$.}
    \label{fig:Z2}
\end{figure}
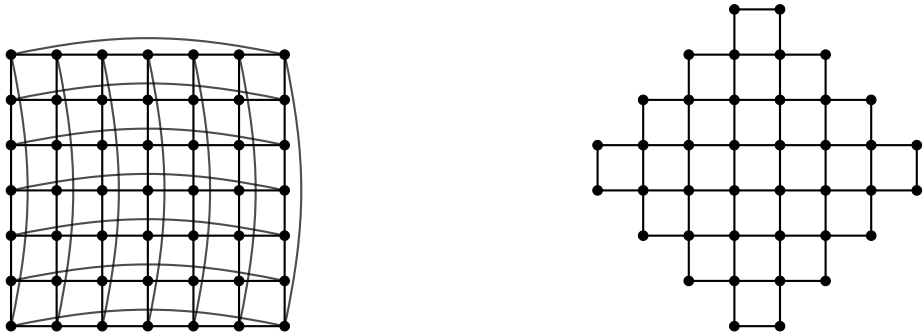

When $L$ is a certain rooted infinite graph with probability $1$
then we simply say that this rooted infinite graph is the limit without any further reference on the distribution.

There are other very natural graph sequences which are Benjamini--Schramm convergent, for instance, $(G_n)_n$ is a sequence of $d$--regular graphs such that the girth $g(G_n)\to \infty$ (length of the shortest cycle), then  it is Benjamini--Schramm convergent and we can even see its limit object: the rooted infinite $d$-regular tree $\mathbb{T}_d$.

There is an alternative way to look at graph parameters that are convergent whenever the graphs are Benjamini--Schramm convergent. For a vertex $v\in V(G)$ let $B_r(v)$ denote its neighborhood of radius $r$. Let $\mathbb{B}_r$ denote all possible $r$-neighborhoods, that is, the rooted graphs of radius at most $r$. We call a bounded graph parameter estimable, if for every $\varepsilon>0$ there are positive integers $k$ and $r$,
and an ``estimator'' function $g: \mathbb{B}_r^k 
 \to \mathbb{R}$ such that for every graph $G$ and uniform,
independently chosen random vertices $v_1,\dots ,v_k \in V (G)$, we have
$$\mathbb{P}(|f(G) - g(B_r(v_1),\dots , B_r(v_k))|>\varepsilon)\leq \varepsilon.$$ 
In other words, $g$ estimates $f$ from a sample chosen according to the rules of sampling from a bounded degree graph. Elek \cite{elek2010parameter} proved that a graph parameter is estimable if and only if it is convergent for every Benjamini--Schramm convergent graph sequence.

\begin{Th}[Elek \cite{elek2010parameter}]\label{thm:elek}
A bounded graph parameter $f$ is estimable if and only if for every Benjamini--Schramm 
convergent graph sequence $(G_n)_n$, the sequence of numbers $(f(G_n))_n$
is convergent.
\end{Th}

So Benjamini--Schramm convergence coincide with a very natural setting for estimating a graph parameter. Also, it is not hard to see that Theorem~\ref{thm:elek} is applicable for graph families that are closed under Benjamini-Schramm convergence, such as Eulerian graphs (we mean in the support of the random rooted graph we only have Eulerian subgraphs).

The main theorem of this paper is the following one.

\begin{Th} \label{th: main} The parameter $\frac{1}{v(G)}\ln \ee(G)$ is estimable for Eulerian graphs, that is, if $(G_n)_n$ is a Benjamini--Schramm convergent sequence of Eulerian graphs, then \\ 
$\lim_{n\to \infty}\frac{1}{v(G_n)}\ln \ee(G_n)$ exists. 
\end{Th}

We remark that for many graph parameter $P(G)$ proving the convergence of \\
$\lim_{n\to \infty}\frac{1}{v(G_n)}\ln P(G_n)$
for subgraphs of lattice graphs is often an easy problem using Fekete's lemma
This is not completely the case for the number of Eulerian orientations as it is sensitive for edge deletion ruining the Eulerian property of the graph. For instance, Baxter \cite{baxter1969f} elaborate on the role of the boundary condition in case of triangle lattice. Theorem~\ref{th: main} shows that the boundary condition is not important as long as the finite graphs are Eulerian. The situation changes dramatically if the boundary condition involves pre-directing some of the edges. For instance, Korepin and Zinn-Justin \cite{korepin2000thermodynamic} showed that the domain wall boundary condition changes the limit value on the square lattice. In fact, it is rather easy to construct examples with pre-directed edges that decreases the number of Eulerian orientations exponentially, see Figure~\ref{fig:boudary_matter} for an example.
Note that Theorem~\ref{th: main} applies for graph sequences too that are not lattice graphs and we cannot even speak about boundary condition at all. 

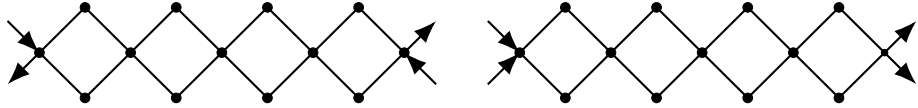
\begin{figure}
\centering
\begin{subfigure}{.4\textwidth}
\centering
    \begin{tikzpicture}[scale=0.6]
        \foreach \x in {0,1,2,3} 
        {
            \draw[thick] (0+2*\x,0) -- (1+2*\x,1) -- (2+2*\x,0) -- (1+2*\x,-1) -- (0+2*\x,0) ;
            \filldraw[black] (2*\x,0) circle (3pt);
            \filldraw[black] (2*\x+1,1) circle (3pt);
            \filldraw[black] (2*\x+1,-1) circle (3pt);
        }
        \filldraw[black] (8,0) circle (3pt);
        \draw[thick,-{Latex[length=3mm]}] (-0.7,0.7) -- (0,0) ;
        \draw[thick,{Latex[length=3mm]}-] (-0.7,-0.7) -- (0,0) ;
    
        \draw[thick,{Latex[length=3mm]}-] (8.7,0.7) -- (8,0) ;
        \draw[thick,-{Latex[length=3mm]}] (8.7,-0.7) -- (8,0) ;
    \end{tikzpicture}        
\end{subfigure}%
\begin{subfigure}{.4\textwidth}
\centering
    \begin{tikzpicture}[scale=0.6]
        \foreach \x in {0,1,2,3} 
        {
            \draw[thick] (0+2*\x,0) -- (1+2*\x,1) -- (2+2*\x,0) -- (1+2*\x,-1) -- (0+2*\x,0) ;
            \filldraw[black] (2*\x,0) circle (3pt);
            \filldraw[black] (2*\x+1,1) circle (3pt);
            \filldraw[black] (2*\x+1,-1) circle (3pt);
        }
        \filldraw[black] (8,0) circle (2pt);
        \draw[thick,-{Latex[length=3mm]}] (-0.7,0.7) -- (0,0) ;
        \draw[thick,-{Latex[length=3mm]}] (-0.7,-0.7) -- (0,0) ;
    
        \draw[thick,{Latex[length=3mm]}-] (8.7,0.7) -- (8,0) ;
        \draw[thick,{Latex[length=3mm]}-] (8.7,-0.7) -- (8,0) ;
    \end{tikzpicture}        
\end{subfigure}
\caption{Two sequences of graphs with different boundary conditions but with the same Benjamini-Schramm limit. The first sequence (as on the left) consists of $n$ cycles with balanced boundary condition at the ends. The second sequence (as on the right) consists also $n$ cycles but at the two end we have $2$ surplus and $2$ deficit. In the first case it is clear that the number of Eulerian orientations is $2^n$, while in the second case it is constant $1$. This example shows that $\tfrac{1}{v(G)}\ln\varepsilon(G)$ is not an estimable parameter for convergent graph sequences with  boundary conditions given by pre-directed edges.}
\label{fig:boudary_matter}
\end{figure}

\section{Proof strategy and preliminaries}

In this section we collect the necessary tools to prove Theorem~\ref{th: main}. We essentially rely on two tools. One of them is the so-called subgraph counting polynomial (that in turn relies on the so-called gauge transformation) and the other one is a Lee-Yang-type theorem on the zeros of a certain polynomial.

The idea of the proof is to encode the number of Eulerian orientations of a graph $G$ as a special evaluation of a certain polynomial. This polynomial will have the property that all its zeros are on the unit circle on the complex plane. For a Benjamini--Schramm convergent graph sequence $(G_n)_n$ the distribution of these zeros will then converge to a limit measure on the unit circle. An extra difficulty of this approach that this special evaluation is exactly the evaluation at $1$, so in principle it can occur that the zeros of the polynomials accumulate at $1$. We overcome this difficulty by a continuity argument using one more special property of the arising polynomials: their coefficients are non-negative. Though this plan might be vague at this moment it will be more clear after the next sections.

\subsection{Subgraph counting polynomial}

For a moment let us assume that $G$ is a $d$-regular graph, and let us introduce the so-called subgraph counting polynomial
$$F_G(x_0,\dots ,x_d)=\sum_{A\subseteq E}\left(\prod_{v\in V}x_{d_A(v)}\right),$$
where $d_A(v)$ is the degree of the vertex $v$ in the graph $G_A=(V,A)$.
And a bit more generally, we can also define
$$F_G(x_0,\dots ,x_d|z)=\sum_{A\subseteq E}\left(\prod_{v\in V}x_{d_A(v)}\right)z^{2|A|}=F_G(x_0,x_1z,x_2z,...,x_dz^d).$$
As an example we give the subgraph counting polynomial $F_{K_5}(x_0,x_1,x_2,x_3,x_4)$ of the complete graph $K_5$ on $5$ vertices. The first term corresponds to the empty subgraph, the last term corresponds to the graph itself.

 \begin{align*}
 &\  x_{0}^{5} + 10 x_{0}^{3} x_{1}^{2} + 15 x_{0} x_{1}^{4} + 30 x_{0}^{2} x_{1}^{2} x_{2} + 30 x_{1}^{4} x_{2} + 60 x_{0} x_{1}^{2} x_{2}^{2} + 10 x_{0}^{2} x_{2}^{3} + 70 x_{1}^{2} x_{2}^{3} + 15 x_{0} x_{2}^{4} \\ 
 &+ 12 x_{2}^{5} + 20 x_{0} x_{1}^{3} x_{3} + 60 x_{1}^{3} x_{2} x_{3} + 60 x_{0} x_{1} x_{2}^{2} x_{3} + 120 x_{1} x_{2}^{3} x_{3} + 60 x_{1}^{2} x_{2} x_{3}^{2} + 30 x_{0} x_{2}^{2} x_{3}^{2} + 70 x_{2}^{3} x_{3}^{2} \\
 & + 60 x_{1} x_{2} x_{3}^{3} + 5 x_{0} x_{3}^{4} + 30 x_{2} x_{3}^{4} + 5 x_{1}^{4} x_{4} + 30 x_{1}^{2} x_{2}^{2} x_{4} + 15 x_{2}^{4} x_{4} + 60 x_{1} x_{2}^{2} x_{3} x_{4} + 60 x_{2}^{2} x_{3}^{2} x_{4} \\
 &+ 20 x_{1} x_{3}^{3} x_{4} + 15 x_{3}^{4} x_{4} + 10 x_{2}^{3} x_{4}^{2} + 30 x_{2} x_{3}^{2} x_{4}^{2} + 10 x_{3}^{2} x_{4}^{3} + x_{4}^{5}.
 \end{align*}

The following theorem connects the number of Eulerian orientations with the subgraph counting polynomial. For a self-contained proof, see the Appendix.
 
\begin{Th}[Borb\'enyi and Csikv\'ari \cite{borbenyi2020counting}] \label{Euler orientations, regular}
For an even number $d$ let $\underline{s}=(s_0,s_1,\dots ,s_d)$ be defined as follows.
$$s_k=\begin{cases} \frac{\binom{d}{d/2}\binom{d/2}{k/2}}{2^{d/2}\binom{d}{k}} & \mbox{if}\ \ k\ \ \mbox{is even}, \\
0  & \mbox{if}\ \ k\ \ \mbox{is odd}.
\end{cases}$$
Then $F_G(s_0,\dots ,s_d)$ counts the number of Eulerian orientations of a $d$--regular graph $G$.	
\end{Th}

If $G$ is not necessarily $d$-regular, then the above definitions have to be changed as follows. For each vertex $v$ we introduce a set of variables $x^{(v)}_0,x^{(v)}_1,\dots ,x^{(v)}_{d(v)}$. Then the subgraph counting function is defined as 
$$F_G\left(\left(x^{(v)}_0,x^{(v)}_1,\dots ,x^{(v)}_{d(v)}\right)_{v\in V}\right)=\sum_{A\subseteq E}\left(\prod_{v\in V}x^{(v)}_{d_A(v)}\right),$$
and
$$F_G\left(\left(x^{(v)}_0,x^{(v)}_1,\dots ,x^{(v)}_{d(v)}\right)_{v\in V}|z\right)=\sum_{A\subseteq E}\left(\prod_{v\in V}x^{(v)}_{d_A(v)}\right)z^{2|A|}.$$
The following generalization of Theorem~\ref{Euler orientations, regular} is also true.

\begin{Th}[Borb\'enyi and Csikv\'ari \cite{borbenyi2020counting}] \label{Euler orientations, non-regular}
Let $G$ be an Eulerian graph. For each vertex $v\in V$ let us introduce the vector $\underline{s}^{(v)}=(s^{(v)}_0,s^{(v)}_1,\dots ,s^{(v)}_{d(v)})$, where 
$$s^{(v)}_k=\begin{cases} \frac{\binom{d(v)}{d(v)/2}\binom{d(v)/2}{k/2}}{2^{d(v)/2}\binom{d(v)}{k}} & \mbox{if}\ \ k\ \ \mbox{is even}, \\
0  & \mbox{if}\ \ k\ \ \mbox{is odd}.
\end{cases}$$
Then $F_G\left(\left(s^{(v)}_0,\dots ,s^{(v)}_{d(v)}\right)_{v\in V}\right)$ counts the number of Eulerian orientations of the graph $G$.	
\end{Th}

\subsection{A Lee-Yang-type theorem: Wagner's subgraph counting technique}

In this section we will recall some theorem of Wagner (Theorem~3.2 of \cite{wagner2009weighted}) about the location of zeros of $F_G(x_0,\dots,x_d|z)$.  For any fixed vertex $v$ and $x^{(v)}_0,\dots,x^{(v)}_{d(v)}$ let us define the following \emph{key-polynomial}
\[
    K_v(x^{(v)}_0,\dots,x^{(v)}_d|z)=\sum_{k=0}^{d(v)} {d(v)\choose k} x^{(v)}_k z^k.
\]

\begin{Th}[Wagner \cite{wagner2009weighted}]\label{thm:wagner}
If for any vertex $v$ the polynomial $K_v(x^{(v)}_0,\dots,x^{(v)}_d|z)$ has no complex zero in the open disk of radius $\kappa$ around 0, then $F_G\left(\left(x^{(v)}_0,x^{(v)}_1,\dots ,x^{(v)}_{d(v)}\right)_{v\in V}|z\right)$ has no complex zero in the open disk of radius $\kappa$ around 0 for any $d$-regular graph $G$.

If for any vertex $v$ the polynomial $K_v(x^{(v)}_0,\dots,x^{(v)}_d|z)$ has no complex zero in the complement of a closed disk of radius $\kappa$ around 0, then $F_G\left(\left(x^{(v)}_0,x^{(v)}_1,\dots ,x^{(v)}_{d(v)}\right)_{v\in V}|z\right)$ has no complex zero in the complement of a closed disk of radius $\kappa$ around 0 for any graph $G$.

In particular, if  for any vertex $v$ the polynomial $K_v(x^{(v)}_0,\dots,x^{(v)}_d|z)$ has only zeros on the circle of radius $\kappa$ around 0, then $F_G\left(\left(x^{(v)}_0,x^{(v)}_1,\dots ,x^{(v)}_{d(v)}\right)_{v\in V}|z\right)$ has complex zeros only on the circle of radius $\kappa$ for any  graph $G$.
\end{Th}

\section{Proof of Theorem~\ref{th: main}}

In this section we complete the proof of Theorem~\ref{th: main}.

For a graph $G$ let us introduce the polynomial
$$P_G(z)=F_G\left(\left(s^{(v)}_0,s^{(v)}_1,\dots ,s^{(v)}_{d(v)}\right)_{v\in V}\bigg|\  z\right),$$
where
$$s^{(v)}_k=\begin{cases} \frac{\binom{d(v)}{d(v)/2}\binom{d(v)/2}{k/2}}{2^{d(v)/2}\binom{d(v)}{k}} & \mbox{if}\ \ k\ \ \mbox{is even}, \\
0  & \mbox{if}\ \ k\ \ \mbox{is odd}.
\end{cases}$$

\begin{Ex}
For the complete graph $K_5$ on $5$ vertices we have
$$P_{K_5}(z)=F_{K_5}\left(\frac{3}{2},0,\frac{1}{2},0,\frac{3}{2} \bigg| \ z\right)=\frac{243}{32}z^{20} + \frac{45}{16}z^{14} + \frac{45}{32}z^{12}   +\frac{3}{8}z^{10} + \frac{45}{32}z^8 + \frac{45}{16}z^6 + \frac{243}{32}.$$
The following picture depicts its zeros.
\begin{figure}[h]
\includegraphics[width=8cm,height=8cm]{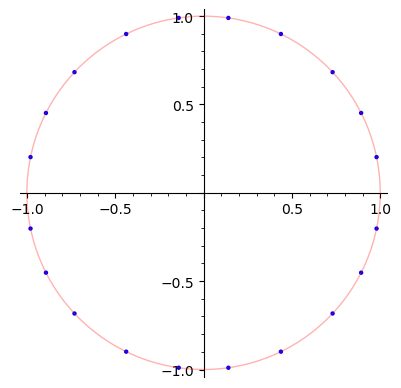}
\caption{The zeros of $P_{K_5}(z)$.}
\end{figure}
\end{Ex}

By Theorem~\ref{Euler orientations, non-regular} we know that $P_G(1)=\ee(G)$. Observe that the polynomials
$$K_v(s^{(v)}_0,\dots,s^{(v)}_d|z)=\sum_{k=0}^{d(v)} {d(v)\choose k} s^{(v)}_k z^k=2^{-d(v)/2}\binom{d(v)}{d(v)/2}(1+z^2)^{d(v)/2},$$
that is, all its zeros lie on the unit circle. By Theorem~\ref{thm:wagner} it implies that the zeros of $P_G(z)$ also lie on the unit circle.
If $G$ has $m$ edges, then the degree of the polynomial $P_G(z)$ is $2m$ and we can factorize it as follows:
$$P_G(z)=2^{-m}\prod_{v\in V}\binom{d(v)}{d(v)/2}\prod_{i=1}^{2m}(z-\rho_i),$$
where $|\rho_i|=1$ for $i=1,\dots ,2m$. Let us introduce the following measure on the complex plane:
$$\mu_G=\frac{1}{2m}\sum_{i=1}^{2m}{\delta}_{\rho_i},$$
where $\delta_s$ is the Dirac-measure supported on $s\in \mathbb{C}$. If $k$ is a fixed non-negative integer, then 
$$\int z^k\ d\mu_{G}(z)=\frac{1}{2m}\sum_{i=1}^k\rho_i^k.$$
If $P_G(z)=\sum_{k=0}^{2m}a_kz^k$, then the integral $\int z^k\ d\mu_{G}(z)$ is determined by the numbers $a_{2m},a_{2m-1},\dots ,a_{2m-k}$
which in turn are determined by the $k$-neighborhood statistics of the graph $G$. It turns out that it implies that if $(G_n)_n$ is a Benjamini--Schramm convergent graph sequence, then the sequence $\int z^k\ d\mu_{G_n}(z)$. The precise details of this argument is given in the paper \cite{csikvari2016benjamini}. A measure sequence $\mu_n$ on $\mathbb{C}$ is convergent if for any fixed $k$ and $\ell$, the sequence $\int z^k\overline{z}^{\ell}\ d\mu_n(z)$ is convergent. Note that $\mu_{G_n}$ is supported on the unit circle, this is equivalent with the convergence of  $\int z^k\ d\mu_{G_n}(z)$. Whence $\mu_{G_n}$ is weakly convergent. 

Now let us fix some $u\neq 1$ positive real number and consider $\frac{1}{v(G_n)}\ln P_{G_n}(u)$. We have
$$\frac{1}{v(G_n)}\ln P_{G_n}(u)=\frac{1}{v(G_n)}\ln \left(2^{-e(G_n)}\prod_{v\in V}\binom{d_{G_n}(v)}{d_{G_n}(v)/2}\right)+\frac{2e(G_n)}{v(G_n)}\int \ln |u-z| \ d\mu_{G_n}(z).$$
Since $\mu_{G_n}$ are supported on the unit circle we get that $h_u(z)=\ln|u-z|$ is a continuous function on an open neighborhood of the unit circle. This gives that the sequence 
$\frac{1}{v(G_n)}\ln P_{G_n}(u)$ exists for $u\neq 1$ positive real number.

Let us introduce 
$$p_L(u)=\lim_{n\to \infty}\frac{1}{v(G_n)}\ln P_{G_n}(u).$$
The final observation is that $p_L(u)$ is a monotone increasing continuous function. This is because $P_G(z)$ has only non-negative coefficients and so if $u_1<u_2$, then
$$P_G(u_1)\leq P_G(u_2)\leq \left(\frac{u_2}{u_1}\right)^{2m}P_G(u_1),$$
whence
$$\frac{1}{v(G)}\ln P_G(u_1)\leq \frac{1}{v(G)}\ln P_G(u_2)\leq \frac{1}{v(G)}\ln P_G(u_1)+\frac{2e(G)}{v(G)}\ln \left(\frac{u_2}{u_1}\right).$$
This implies that 
$$p_L(u_1)\leq p_L(u_2)\leq p_L(u_1)+\Delta \ln \left(\frac{u_2}{u_1}\right)$$
showing that $p_L(u)$ is a continuous and monotone increasing function. In particular, we can introduce $p_L(1)=\lim_{u\to 1}p_L(u)$ and get that
$$\lim_{n\to \infty}\frac{1}{v(G_n)}\ln P_{G_n}(1)=p_L(1),$$
that is, $\lim_{n\to \infty}\frac{1}{v(G_n)}\ln \ee(G_n)$ exists.

\section{Concluding remarks}

In this last section we give some remarks on the methods used in this paper.

\subsection{Large girth graphs}
In this section we determine the limit of $\frac{1}{v(G_n)}\ln \ee(G_n)$ if $(G_n)_n$ is a large girth sequence, that is, $g(G_n)\to \infty$. This limit was determined by Vergnas \cite{vergnas1990upper} building on the work of Schrijver \cite{schrijver1983bounds} if $(G_n)_n$ is a sequence of $d$-regular graphs. Indeed, Schrijver proved the lower bound
$$\frac{1}{v(G)}\ln \ee(G)\geq \ln\left(2^{-d/2}\binom{d}{d/2}\right),$$
and Vergnas proved a matching upper bound in terms of the maximal number of pairwise edge-disjoint cycles which is at most $\frac{dv(G)}{g}$ if $g$ the length of the shortest cycle. 

Here we directly rely on the proof method we did in the previous section.

\begin{Th} \label{large-girth}
Let $(G_n)_n$ be a Benjamini--Schramm convergent sequence of Eulerian graphs with maximum degree $\Delta$ and  girth $g(G_n)\to \infty$. Let
$$t_k:=\lim_{n\to \infty}\frac{|\{v\ |\ d_{G_n}(v)=k\}|}{v(G_n)}\  \  \ (k=0,\dots ,\Delta),$$
then
$$\lim_{n\to \infty}\frac{1}{v(G_n)}\ln \ee(G_n)=\sum_{k=0}^{\Delta}t_k\ln \left(2^{-k/2}\binom{k}{k/2}\right).$$
\end{Th}

\begin{proof} Recall that for positive real number $u\neq 1$ we had the formula
$$\frac{1}{v(G_n)}\ln P_{G_n}(u)=\frac{1}{v(G_n)}\ln \left(2^{-e(G_n)}\prod_{v\in V}\binom{d_{G_n}(v)}{d_{G_n}(v)/2}\right)+\frac{2e(G_n)}{v(G_n)}\int \ln |u-z| \ d\mu_{G_n}(z).$$
Here the first term converges to 
$$\sum_{k=0}^{\Delta}t_k\ln \left(2^{-k/2}\binom{k}{k/2}\right).$$
We only need to understand the second term. In particular, we need to understand the limit of the measures $\mu_{G_n}$. We claim that this limit measure is the uniform measure on the unit circle. We claim that $P_G(z)=\sum_{k=0}^{2m}a_kz^k$, then $a_{2m-1}=a_{2m-2}=\dots =a_{2m-2g+1}=0$ if the girth is bigger than $g$. Since $a_k=a_{2m-k}$ by the symmetric nature of the vectors $\underline{s}^{(v)}$ we only need to see that $a_1=\dots=a_{2g-1}=0$ which follows since if $A\subseteq E$ satisfies that $0<|A|<g$, then there is a vertex $v$ such that $d_A(v)=1$, and then $\prod_{v\in V}s^{(v)}_{d_A(v)}=0$. From the Newton-Waring formulas we also get that $\int z^k d \mu_G(z)=0$ for $k=1,\dots ,2g-1$. Since $g(G_n)\to \infty$ we get that for the limit measure $\mu_L$ we have $\int z^k\ d \mu_L(z)=0$ for every integer $k\geq 1$. Hence $\mu_L$ is the uniform measure, and
$$\lim_{u\to 1}\int \ln(u-z) \ d\mu_L(z)=0.$$
This completes the proof.

\end{proof}

\begin{Rem}
One can prove that for any $g$ and $\alpha>0$ there exists a $C(d, g,\alpha)>0$ such that
if the $d$-regular graph $G$ contains more than $\alpha v(G)$ cycles of length at most $g$, then
$$\frac{1}{v(G)}\ln \varepsilon(G)>\ln\left(\frac{\binom{d}{d/2}}{2^{d/2}}\right)+C(d, g,\alpha).$$
In other words, if $(G_n)_n$ is a sequence of $d$-regular graphs such that $G_n\not\to \mathbb{T}_d$, then 
$$\limsup_{n\to \infty}\frac{1}{v(G)}\ln \varepsilon(G)>\ln\left(2^{-d/2}\binom{d}{d/2}\right).$$
This statement can be seen from the subgraph counting polynomial.

Theorem~\ref{large-girth} also follows from Theorem 5.1 of \cite{isaev2024correlation}.

\end{Rem}

\subsection{What goes wrong with perfect matchings?}

To have a better understanding of the proof strategy used in this paper we carefully analyze another graph invariant in this section, namely, the number of perfect matchings, hereafter denoted by $\mathrm{pm}(G)$.

Clearly, if we have a graph with a lot of perfect matchings, and we delete one vertex the number of perfect matchings drops to zero. This means that we need to impose some restriction on the graph class. Note that even in the case of Eulerian orientations we needed to require that the elements of the graph sequence $(G_n)_n$ are Eulerian graphs. Unfortunately, even with the assumption that all $G_n$ are $d$-regular bipartite graphs one can construct a sequence of graphs $(G_n)_n$ such that $\frac{1}{v(G_n)}\ln \mathrm{pm}(G_n)$ is not convergent \cite{abert2016matchings}. Nevertheless, there is one positive result: it is convergent if $G_n$ are not only $d$-regular bipartite graphs, but $g(G_n)\to \infty$ is also satisfied \cite{abert2016matchings}. 

It is very instructive to see what goes wrong in the case of the number of perfect matchings in our proof. Suppose for simplicity that $G_n$ are $4$-regular graphs. Then $\mathrm{pm}(G)=F_G(0,1,0,0,0)$ by the definition of the subgraph counting polynomial. This would not be very useful as $F_G(0,1,0,0,0|z)=\mathrm{pm}(G)z^{v(G)}$. Fortunately, $F_G(x_0,x_1,x_2,x_3,x_4)$ takes the same value at several different places due to some invariance under ``rotations'', see details in \cite{borbenyi2020counting}. In particular,
$$F_G(0,1,0,0,0)=F_G\left(1,-\frac{1}{2},0,\frac{1}{2},-1\right).$$
For this vector we have
$$K_v(z)=1-2z+2z^3-z^4=(1-z)^3(1+z),$$
so all zeros have absolute value $1$. (There is always such a vector for $(0,1,0,\dots,0)$ no matter what $d$ is.) This means that
$$F_G\left(1,-\frac{1}{2},0,\frac{1}{2},-1\bigg|\ z\right)$$
have all zeros lying on the unit circle. It even implies that the function \\ $P_G(u):=F_G\left(1,-\frac{1}{2},0,\frac{1}{2},-1|u\right)$ is non-negative for real $u>1$ implying that for such a $u$ the
$$p_L(u):=\lim_{n\to \infty}\frac{1}{v(G_n)}\ln P_{G_n}(u)$$
exists. If $\mathrm{pm}(G_n)\neq 0$, then $P_{G_n}(1)\neq 0$ and we can also deduce that $P_{G_n}(u)>0$ for $0<u<1$ so $p_L(u)$ exists in this case. Unfortunately, since the coefficients of $P_G(z)$ are not necessarily non-negative  we cannot argue that it is monotone increasing, and that $p_L(u)$ is continuous at $1$. 

Though this strategy does not work in the case of perfect matchings, it is still instructive to see how gauge transformation gives us a great flexibility to choose the vectors in such a way that we can apply a Lee-Yang-type theorem.

\bibliography{bibliography}
\bibliographystyle{plain}

\newpage

\section{Appendix: Eulerian orientations via subgraph counting polynomial}

In this appendix we give a self-contained proof of Theorem~\ref{Euler orientations, regular}. First we introduce the so-called normal factor graph and gauge transformation, then we prove the aforementioned theorem.

\subsection{Normal factor graphs and gauge transformations}

The following concept will enable us to encode the number of Eulerian orientations and the subgraph counting polynomial in a unified framwework.

\begin{Def} A normal factor graph $\mathcal{H}=(V,E,\mathcal{X},(f_v)_{v\in V})$ is a graph $(V,E)$ equipped with an alphabet $\mathcal{X}$ and  a function $f_v: \mathcal{X}^{d_v}\to \mathbb{R}$ at each vertex. At each edge $e$ there is a variable $x_e$ taking values from the alphabet $\mathcal{X}$. The partition function 
$$Z(\mathcal{H})=\sum_{\sigma\in \mathcal{X}^E}\prod_{v\in V}f_v(\sigma_{\partial v}),$$ 
where $\sigma_{\partial v}$ is the restriction of $\sigma$ to the the edges incident to the vertex $v$. 
\end{Def}

For instance, if $\mathcal{X}=\{0,1\}$ and 
$$f_v(\sigma_1,\dots ,\sigma_{d_v})=\left\{ \begin{array}{cl} 1 & \mbox{if}\ \sum_{i=1}^{d_v}\sigma_i=1, \\ 0 & \mbox{otherwise}, \end{array} \right.$$
where $d_v$ is the degree of the vertex $v$, 
then $Z(\mathcal{H})$ is exactly the number of perfect matchings of the underlying graph. 

Let $\mathcal{H}=(V,E,\mathcal{X},(f_v)_{v\in V})$ be a normal factor graph with alphabet $\mathcal{X}$. We will show that it is possible to introduce a new normal factor graph $\widehat{\mathcal{H}}=(V,E,\mathcal{Y},(\widehat{f_v})_{v\in V})$ on the same graph with new functions $\widehat{f_v}$ and alphabet $\mathcal{Y}$ such that $Z(\widehat{\mathcal{H}})=Z(\mathcal{H})$. As we will see, sometimes it will be more convenient to study the new normal factor graph $\widehat{\mathcal{H}}$. 

Let $\mathcal{Y}$ be a new alphabet, and for each edge $(u,v)\in E$ let us introduce two new matrices, $G_{uv}$ and $G_{vu}$ of size $\mathcal{Y}\times \mathcal{X}$. The new variables will be denoted by $\tau\in \mathcal{Y}^E$, the old ones by $\sigma\in \mathcal{X}^E$. For a vertex $v$ with degree $d_v=k$ let
$$\widehat{f_v}(\tau_{vu_1},\dots ,\tau_{vu_k})=\sum_{\sigma_{vu_1},\dots ,\sigma_{vu_k}}\left(\prod_{u_i \in N(v)}G_{vu_i}(\tau_{vu_i},\sigma_{vu_i})\right)f_v(\sigma_{vu_1},\dots ,\sigma_{vu_k}).$$
This way we defined the functions $\widehat{f_v}$ of $\widehat{\mathcal{H}}$. 
\medskip

This transformation is called a gauge transformation. In computer science, this method was introduced by  Valiant under the name holographic reduction \cite{valiant2008holographic,valiant2006accidental,valiant2002quantum,valiant2002expressiveness}. In statistical physics, it was developed by Chertkov and  Chernyak under the name gauge transformation \cite{chertkov2006loop2,chertkov2006loop1}. Wainwright, Jaakola, Willsky had a related idea under the name reparametrization \cite{wainwright2003tree}, but it is not easy to see the connection. In the different cases the scope was slightly different, Valiant used it as a reduction method for computational complexity of counting problems. This line of research was extended in a series of papers of Jin-Yi Cai and his coauthors, see Jin-Yi Cai's book \cite{cai2017complexity} and the papers \cite{cai2007holographic,cai2008basis,cai2007symmetric,cai2008holographic, cai2011holographic, cailu2008holographic} and references therein. Chertkov and Chernyak \cite{chertkov2006loop2,chertkov2006loop1} studied the so-called Bethe--approximation through gauge transformations. We simply use it as a method of proving the identities like Theorem~\ref{Euler orientations, non-regular}.

The following theorem is due to Chertkov and Chernyak \cite{chertkov2006loop2,chertkov2006loop1} and independently Valiant \cite{valiant2008holographic}. 

\begin{Th} If for each edge $(u,v)\in E$ we have $G^T_{uv}G_{vu}=\mathrm{Id}_{\mathcal{X}}$, then $Z(\widehat{\mathcal{H}})=Z(\mathcal{H})$. 
\end{Th}

\begin{proof}
Let us start to compute 
$Z(\widehat{\mathcal{H}})=\sum_{\tau\in \mathcal{Y}^E}\prod_{v\in V}\widehat{f_v}(\tau_{\partial v})$:
$$Z(\widehat{\mathcal{H}})=\sum_{\tau\in \mathcal{Y}^E}\prod_{v\in V}\left[\sum_{\sigma_{vu_1},\dots ,\sigma_{vu_k}}\left(\prod_{u_i \in N(v)}G_{vu_i}(\tau_{vu_i},\sigma_{vu_i})\right)f_v(\sigma_{vu_1},\dots ,\sigma_{vu_k})\right].$$
If we expand it will have terms $\prod_{v\in V}f_v(\sigma_{vu_1},\dots ,\sigma_{vu_k})$ with some coefficients. A priori it can occur that these terms are incompatible in the sense that $\sigma_{uv}\neq \sigma_{vu}$. As we will see, the role of the conditions on $G_{uv}$ is exactly to ensure that if there is an edge $(u,v)\in E$ with 
$\sigma_{uv}\neq \sigma_{vu}$, then the coefficient is $0$, and if all edges are compatible, then the coefficient is $1$. Indeed, the coefficient is
$$\sum_{\tau\in \mathcal{Y}^E}\prod_{v\in V}\prod_{u_i \in N(v)}G_{vu_i}(\tau_{vu_i},\sigma_{vu_i}).$$
Note that $\tau_{uv}=\tau_{vu}$ for each edge, and this variable appears only at the vertices $u$ and $v$, and nowhere else. Hence
$$\sum_{\tau\in \mathcal{Y}^E}\prod_{v\in V}\prod_{u_i \in N(v)}G_{vu_i}(\tau_{vu_i},\sigma_{vu_i})=\prod_{(u,v)\in E}\left(\sum_{\tau_{uv}}G_{uv}(\tau_{uv},\sigma_{uv})G_{vu}(\tau_{vu},\sigma_{vu})\right)=$$
$$=\prod_{(u,v)\in E}\left(\sum_{\tau_{uv}}G^T_{uv}(\sigma_{uv},\tau_{vu})G_{vu}(\tau_{vu},\sigma_{vu})\right)=\prod_{(u,v)\in E}(G^T_{uv}G_{vu})_{\sigma_{uv},\sigma_{vu}}=
\prod_{(u,v)\in E}(\mathrm{Id})_{\sigma_{uv},\sigma_{vu}}.$$
Hence this is only non-zero if $\sigma_{uv}=\sigma_{vu}$ for each edge $(u,v)\in E(G)$, and then this coefficient is $1$.
\end{proof}

\subsection{Eulerian orientations}

In this section we prove Theorem~\ref{Euler orientations, regular}. For sake of convenience we repeat the theorem.
\bigskip

\noindent \textbf{Theorem~\ref{Euler orientations, regular}} 
{\textit{Let $\underline{s}=(s_0,s_1,\dots ,s_d)$ be defined as follows.
$$s_k=\left\{ 
\begin{array}{cc} \frac{\binom{d}{d/2}\binom{d/2}{k/2}}{2^{d/2}\binom{d}{k}} & \mbox{if}\ \ k\ \ \mbox{is even}, \\
0  & \mbox{if}\ \ k\ \ \mbox{is odd}.
\end{array} \right.$$
Then $F_G(s_0,\dots ,s_d)$ counts the number of Eulerian orientations of a $d$--regular graph $G$.}}
\bigskip

\begin{proof}
First we encode the number of Eulerian orientations as a partition function of a normal factor graph. Let $\mathrm{Sub}(G)$ be the subdivision of the graph $G$, that is, we put a vertex to every edge. The vertex set of $\mathrm{Sub}(G)$ naturally correspond to $V\cup E$, where $G=(V,E)$. An orientation of $G$ correspond to an edge configuration of $\mathrm{Sub}(G)$, where each edge $e\in V(\mathrm{Sub}(G))$ is incident to exactly one edge: a directed edge $(v,u)$ corresponds a configuration, where $(v,e_{v,u})$ belongs to the configuration, but $(u,e_{v,u})$ does not. So we can describe an Eulerian orientation with the local functions
$$f_v(\sigma_{v,e_{v,u_1}},\dots ,\sigma_{v,e_{v,u_d}})=
\begin{cases}
1 & \mbox{if}\ \ \sum_{u_i\in N_G(v)}\sigma_{ve_{v,u}}=d/2, \\
0 & \mbox{if}\ \ \sum_{u_i\in N_G(v)}\sigma_{ve_{v,u}}\neq d/2.
\end{cases}
$$
and
$$f_{e_{u,v}}(\sigma_{u,e_{u,v}},\sigma_{v,e_{u,v}})=
\begin{cases}
1 & \mbox{if}\ \ \sigma_{u,e_{u,v}},\sigma_{v,e_{u,v}}=1, \\
0 & \mbox{if}\ \ \sigma_{u,e_{u,v}},\sigma_{v,e_{u,v}}\neq 1.
\end{cases}
$$
Next we use the gauge theory. For each edge $e=(u,v)\in E(G)$ we introduce two matrices in $\mathrm{Sub}(G)$:
$G_{eu}=G_{ev}=G_1$ and $G_{ue}=G_{ve}=G_2$, where
$$G_1:=\frac{1}{\sqrt{2}}\left( \begin{array}{cc} 1 & 1 \\ i& -i \end{array}\right)\ \ \ \mbox{and} \ \ \ G_2:=\frac{1}{\sqrt{2}}\left( \begin{array}{cc} 1 & 1 \\ -i & i \end{array}\right).$$
In what follows  the rows and columns of $G_1,G_2,F_e$ are indexed by $0$ and $1$, and for a matrix $A$ and $\sigma,\tau\in \{0,1\}$ we use the notation $A(\sigma,\tau)$ for the corresponding element. In particular, we have
$$F_e=\left( \begin{array}{cc} 0 & 1 \\ 1& 0 \end{array}\right)$$
Observe that $G_2^TG_1=\mathrm{Id}$. First let us compute $\widehat{f_e}(\tau_1,\tau_2)$:
\begin{align*}\widehat{f_e}(\tau_1,\tau_2)&=\sum_{\sigma_1,\sigma_2}G_1(\tau_1,\sigma_1)G_1(\tau_2,\sigma_2)f_e(\sigma_1,\sigma_2)\\
&=\sum_{\sigma_1,\sigma_2}G_1(\tau_1,\sigma_1)f_e(\sigma_1,\sigma_2)G^T_1(\sigma_2,\tau_2)\\
&=(G_1F_eG^T_1)(\tau_1,\tau_2).
\end{align*}
Hence by simple matrix multiplication we have
$$\widehat{F_e}=G_1F_eG^T_1=\left( \begin{array}{cc} 1 & 0 \\ 0 & 1 \end{array}\right).$$
This means that in $Z(\widehat{\mathcal{H}})$ only those terms will survive that correspond to a subgraph of $G$. 

Next let us compute $\widehat{f_v}(\tau_1,\dots ,\tau_{d})$. By definition
$$\widehat{f_v}(\tau_1,\dots ,\tau_{d})=\sum_{\sigma_1,\dots ,\sigma_{d}}\prod_{i=1}^{d}G_2(\tau_i,\sigma_i)f_v(\sigma_1,\dots ,\sigma_{d}).$$
Recall that only those terms remain, where $\sum_i\sigma_i=\frac{d}{2}$. Suppose that $\sum_i\tau_i=k$. If there are $j$ places where both $\sigma_i=\tau_i=1$, then its contribution to the sum is $i^j(-i)^{k-j}$, so
$$\widehat{f_v}(\tau_1,\dots ,\tau_{d})=\sum_{\sigma_1,\dots ,\sigma_{d}}\prod_{i=1}^{d}G_2(\tau_i,\sigma_i)f_v(\sigma_1,\dots ,\sigma_{d})=\frac{1}{2^{d/2}}\sum_{j=0}^k\binom{k}{j}\binom{d-k}{d/2-j}(-1)^{k-j}i^k.$$
Observe that
\begin{align*}
\sum_{j=0}^{d/2}(-1)^{k-j}\binom{k}{j}\binom{d-k}{d/2-j}&=\sum_{j=0}^{d/2}(-1)^{k-j}\frac{k!(d-k)!}{j!(k-j)!(d/2-j)!(d/2-k+j)!}\\
&=\frac{\binom{d}{d/2}}{\binom{d}{k}}\sum_{j=0}^{d/2}(-1)^{k-j}\binom{d/2}{j}\binom{d/2}{k-j}.
\end{align*}
Note that $\sum_{j=0}^{d/2}(-1)^{k-j}\binom{d/2}{j}\binom{d/2}{k-j}$ is the coefficient of $x^k$ in  
$$(1-x)^{d/2}(1+x)^{d/2}=(1-x^2)^{d/2}$$
which is clearly $0$ if $k$ is odd, and $(-1)^{k/2}\binom{d/2}{k/2}$ if $k$ is even. Hence
$$\widehat{f_v}(\tau_1,\dots ,\tau_{d})=s_{||\tau||_1}.$$
This means that 
$$Z(\widehat{\mathcal{H}})=F_G(s_0,\dots ,s_d).$$
\end{proof}

\begin{Rem}
Let $G$ be a regular graph, and let
$$H_G(y_{-d},y_{-d+2},\dots ,y_{d-2},y_d)=\sum_{\mathcal{O}}\prod_{v\in V}y_{d^+_{\mathcal{O}(v)}-d^-_{\mathcal{O}}(v)},$$
where the summation goes for all orientations of the graph $G$, and $d^+_{\mathcal{O}(v)}$ and $d^-_{\mathcal{O}}(v)$ are the out-degree and in-degree of $G$ in $\mathcal{O}$. In general it is true that there is a $(d+1)\times (d+1)$ matrix $M_d$ such that
$H_G(\underline{x})=F_G(M_d\underline{x})$ for every graph $G$ and $\underline{x}\in \mathbb{C}^{d+1}$. This statement also extends to non-regular graphs.

\end{Rem}

\end{document}